\documentclass[11pt]{amsart}

\usepackage{amsmath}
\usepackage{amssymb}
\usepackage{amscd}
\usepackage{latexsym}
\usepackage{graphicx}
\usepackage{color}

\newcommand{\ct}{\mbox{{\sc c}}}

\newcommand{\ispa}[1]{\langle \,#1 \,\rangle }

\newcommand{\ol}{\overline}

\newcommand{\mb}{\mathbb}
\newcommand{\acal}{\mathcal{A}}

\newcommand{\hcal}{\mathcal{H}}

\newcommand{\ep}{\epsilon}

\newcommand{\re}{{\rm Re}\,}

\newcommand{\wlim}{\operatornamewithlimits{{\rm w}-lim}}

\newtheorem{thm}{{\sc Theorem}}

\newtheorem{lem}[thm]{{\sc Lemma}}

\newenvironment{example}{\medskip\noindent{\it Example:\/} }{$\square$\par\medskip}

\begin{document}

\title[An algebraic structure of one-dimensional quantum walks]
{An algebraic structure for one-dimensional quantum walks
and a new proof of the weak limit theorem}
\author{Tatsuya Tate}
\address{Graduate School of Mathematics, Nagoya University, 
Furo-cho, Chikusa-ku, Nagoya, 464-8602 Japan}
\email{tate@math.nagoya-u.ac.jp}
\address{{\it Current address}: Mathematical Institute, Graduate School of Sciences, Tohoku University, Aoba, Sendai 980-8578 Japan}
\email{tate@math.tohoku.ac.jp}
\thanks{The author is partially supported by JSPS Grant-in-Aid for Scientific Research (No. 21740117).}
\date{\today}

\renewcommand{\thefootnote}{\fnsymbol{footnote}}
\renewcommand{\theequation}{\arabic{equation}}
\renewcommand{\labelenumi}{{\rm (\arabic{enumi})}}
\renewcommand{\labelenumii}{{\rm (\alph{enumii})}}

\maketitle

\begin{abstract}
An algebraic structure for one-dimensional quantum walks is introduced. 
This structure characterizes, in some sense, one-dimensional quantum walks. 
A natural computation using this algebraic structure leads us to obtain an effective formula for the 
characteristic function of the transition probability. 
Then, the weak limit theorem for the transition probability of quantum walks is deduced 
by using simple properties of the Chebyshev polynomials. 
\end{abstract}



\section{Introduction}\label{INTRO}
The notion of quantum walks was introduced by Aharonov-Davidovich-Zagury\cite{ADZ} as a quantum analogue 
of classical random walks and then re-discovered in computer science. 
In particular, Ambainis-Kempe-Rivosh\cite{AKR} utilized two-dimensional quantum walks to improve Grover's quantum search algorithm. 
See \cite{Ke}, \cite{Ko2} for historical backgrounds and various aspects of quantum walks. 
Among numerous works, Konno\cite{Ko1} obtained a weak limit formula for 
transition probabilities of quantum walks on the one-dimensional integer lattice. 
In this short paper, we present a simple proof of Konno's weak limit theorem. 
First of all, let us state Konno's theorem. 
The quantum walks we consider in this paper is defined in terms of a two-by-two (special) unitary matrix
\begin{equation}\label{Amat}
A=
\begin{pmatrix}
a & b \\
-\ol{b} & \ol{a}
\end{pmatrix},\quad 
a,b \in \mb{C},\ |a|^{2}+|b|^{2}=1,  
\end{equation}
and its decomposition, 
\[
A=P_{A}+Q_{A},\quad 
P_{A}=
\begin{pmatrix}
a & 0 \\
-\ol{b} & 0
\end{pmatrix},\quad 
Q_{A}=
\begin{pmatrix}
0 & b \\
0 & \ol{a}
\end{pmatrix}. 
\]
Let $\ell^{2}(\mb{Z}) \otimes \mb{C}^{2}$ be the Hilbert space 
of square summable functions on $\mb{Z}$ with values in $\mb{C}^{2}$ whose inner product is given by 
\[
\ispa{f,g}=\sum_{x \in \mb{Z}}\ispa{f(x),g(x)}_{\mb{C}^{2}} \quad 
(f,g \in \ell^{2}(\mb{Z}) \otimes \mb{C}^{2}), 
\] 
where $\ispa{\cdot,\cdot}_{\mb{C}^{2}}$ denotes the standard inner product on $\mb{C}^{2}$. 
For any $u \in \mb{C}^{2}$ and $x \in \mb{Z}$, define $\delta_{x} \otimes u \in \ell^{2}(\mb{Z}) \otimes \mb{C}^{2}$ by 
\[
(\delta_{x} \otimes u)(y)=
\begin{cases}
u & (y=x),\\
0 & (y \neq x). 
\end{cases}
\]
Then, the vectors, $\delta_{x} \otimes {\bf e}_{i}$ $(i=1,2,\ x \in \mb{Z})$, where $\{{\bf e}_{1}, {\bf e}_{2}\}$ denotes the standard 
orthonormal basis in $\mb{C}^{2}$, form a complete orthonormal system in $\ell^{2}(\mb{Z}) \otimes \mb{C}^{2}$. 
The unitary evolution, $U(A)$, for the quantum walk is a unitary operator on $\ell^{2}(\mb{Z}) \otimes \mb{C}^{2}$ defined as 
\begin{equation}\label{qwalk}
U(A)=P_{A}\tau+Q_{A}\tau^{-1}, 
\end{equation}
where $\tau$ is the shift operator on $\ell^{2}(\mb{Z}) \otimes \mb{C}^{2}$ defined by $\tau(\delta_{x} \otimes u)=\delta_{x+1} \otimes u$. 
Since the operator $U(A)$ is unitary, the function 
\begin{equation}\label{df0}
p_{n}^{A}(\varphi;x):=\|U(A)^{n}(\delta_{0} \otimes \varphi)(x)\|_{\mb{C}^{2}}^{2} 
\quad (x \in \mb{Z}), 
\end{equation}
defines a probability distribution on $\mb{Z}$ for each positive integer $n$ and a unit vector $\varphi$ in $\mb{C}^{2}$.  
We call the distribution $p_{n}^{A}(\varphi;x)$ the transition probability of the quantum walk $U(A)$.\footnote{In \cite{Ko1}, the `row' decomposition of the given matrix $A$ is used. 
In \cite{ST}, the unitary operator $P_{A}\tau^{-1}+Q_{A}\tau$ is used. In each case, suitable change of the initial data or the change of the sign of the variable $x$ 
will give us our distribution.} 
In this context, Konno's weak limit theorem is stated as follows. 

\begin{thm}\label{KonnoWL}
Let $\varphi=\,^{t}(\varphi_{1},\varphi_{2})$ be a unit vector in $\mb{C}^{2}$. 
Suppose that the components, $a$, $b$, of the matrix $A$ given in $\eqref{Amat}$ are non-zero. 
Then, we have the following weak limit formula.  
\[
\wlim_{n \to \infty}\sum_{x \in \mb{Z}}p_{n}^{A}(\varphi;x)\delta_{x/n}=\chi_{(-|a|,|a|)}(y) \frac{|b|(1+\lambda_{A}(\varphi)y)}{\pi (1-y^{2}) \sqrt{|a|^{2}-y^{2}}}, 
\]
where $\delta_{x/n}$ denotes the Dirac measure at $x/n$, $\chi_{(-|a|,|a|)}(y)$ is the characteristic function on the interval $(-|a|,|a|)$ and the constant $\lambda_{A}(\varphi)$ is given by 
\[
\lambda_{A}(\varphi)=|\varphi_{1}|^{2}-|\varphi_{2}|^{2}-\frac{1}{|a|^{2}}(ab\ol{\varphi}_{1}\varphi_{2}+\ol{a}\ol{b} \varphi_{1} \ol{\varphi}_{2}). 
\]
\end{thm}
There are several methods to prove Theorem $\ref{KonnoWL}$. In \cite{Ko1}, Konno computed the probability distribution $p_{n}^{A}(\varphi;x)$ directly and explicitly, 
and then used a result on asymptotic behavior of Jacobi polynomials\cite{CI}. 
In \cite{GJS}, Grimmett-Janson-Scudo employs an integral formula involving the eigenvalues 
of the matrix
\[
A(z)=
\begin{pmatrix}
az & bz^{-1} \\
-\ol{b}z & \ol{a}z^{-1}
\end{pmatrix},\quad z \in \mb{C} \setminus \{0\}. 
\]
Explicit computation of $p_{n}^{A}(\varphi;x)$ in \cite{Ko1} involves computation of words in the matrix $P_{A}$, $Q_{A}$. 
It is a very interesting computation, but it is a bit complicated. 
Furthermore, the asymptotic formula for Jacobi polynomials used in \cite{Ko1} is not elementary. 
The integral formula involving the eigenvalues of the matrix $A(z)$ is very powerful. 
Indeed, in \cite{ST}, the integral formula is used to compute the local asymptotics of $p_{n}^{A}(\varphi;x)$. 
However, the integral formula does not make explicit form of $p_{n}^{A}(\varphi;x)$ completely clear. 
Furthermore, it would not be quite easy to compute the eigenvalues in higher dimensional cases. 

In the present paper, we give a simple and systematic computation of $p_{n}^{A}(\varphi;x)$ and prove 
Theorem $\ref{KonnoWL}$ in an elementary way. 
Our strategy is to exploit an algebraic structure behind the quantum walk $U(A)$. 
See Section $\ref{ALG}$ for details on this algebraic structure. 
This structure characterizes, in some sense, the quantum walk $U(A)$. 
Furthermore, this structure naturally leads us to obtain an effective formula of $p_{n}^{A}(\varphi;x)$ in a systematic way, 
and then asymptotic behavior of the characteristic function of the distribution $p_{n}^{A}(\varphi;x)$ 
is easily deduced using simple properties of the Chebyshev polynomials. 
In the present paper, higher dimensional cases are not discussed. 
(See \cite{WKKK} for a pioneer work on the weak limit formula of certain one-parameter family of two-dimensional quantum walks.) 
However, we expect that the algebraic structure introduced here would give some insight for quantum walks in higher dimension.

\section{An algebraic structure for one-dimensional quantum walks}\label{ALG}

Suppose that we are given unitary operators, 
\[
V,\quad W,\quad \sigma, 
\]
on a Hilbert space $\hcal_{0}$, whose inner product is denoted by $\ispa{\cdot,\cdot}$, such that they satisfy the following conditions. 

\begin{itemize}
\item[(QW1)] $W^{2}=-I$. 
\item[(QW2)] $VW=WV^{-1}$. 
\item[(QW3)] $\sigma W +W \sigma =\sigma V- V \sigma =0$. 
\item[(QW4)] $\sigma^{*}=\sigma$. 
\end{itemize}

Note that (QW3) implies that $\sigma \neq I$.  
Since $W$, $\sigma$ are unitary operators, (QW1), (QW4) imply 
$W^{*}=W^{-1}=-W$, $\sigma^{*}=\sigma^{-1}=\sigma$. 
Let 
\[
\pi_{\pm}:=\frac{1}{2}(I \pm \sigma)
\]
be the projection onto the eigenspace of $\sigma$ with the eigenvalue $\pm 1$. 
Since $\sigma \neq I$, we have $\pi_{-} \neq 0$ and (QW3) implies that 
$V \pi_{\pm}=\pi_{\pm} V$, $W\pi_{\pm}=\pi_{\mp}W$. 
We set 
\begin{equation}\label{shift1}
X=\frac{1}{2}(V+V^{*}),\quad Y=\frac{1}{2i}(V-V^{*}),\quad T=X+i\sigma Y. 
\end{equation}
Then, it is obvious that we have
\begin{equation}\label{rel1}
\begin{gathered}
XY=YX,\quad XW=WX,\quad YW+WY=0,\quad VT=TV,\quad TW=WT,\\
X \sigma =\sigma X,\quad Y\sigma =\sigma Y,\quad T \sigma =\sigma T. 
\end{gathered}
\end{equation}
Since $T^{*}=X-i\sigma Y$, we see 
\[
T^{*}T=TT^{*}=X^{2}+Y^{2}=I, 
\]
which indicates that $T$ is a unitary operator on $\hcal_{0}$. 
The following lemma can easily be proved. 
\begin{lem}\label{auxL1}
\noindent$(1)$ We have $T=\pi_{+}V+\pi_{-}V^{*}$, $V=\pi_{+}T+\pi_{-}T^{*}$. 

\noindent$(2)$ Set $\ep =VW$. Then we have 
\[
\ep^{*}=\ep^{-1}=-\ep,\quad \ep \pi_{\pm}=\pi_{\mp}\ep,\quad 
\ep W=-V,\quad W\ep =-V^{*},\quad \ep V=V^{*}\ep,\quad \ep \sigma +\sigma \ep =0.
\]
\end{lem}

\begin{example}\label{ex0}
We take $\alpha,\beta \in \mb{C}$ with $|\alpha|=|\beta|=1$ and define the matrices, $V_{0}$, $W_{0}$, $\sigma$, by 
\[
V_{0}=
\begin{pmatrix}
\alpha & 0 \\
0 & \ol{\alpha}
\end{pmatrix},
\quad 
W_{0}=
\begin{pmatrix}
0 & \beta \\
-\ol{\beta} & 0 
\end{pmatrix},\quad 
\sigma =
\begin{pmatrix}
1 & 0 \\
0 & -1
\end{pmatrix}. 
\]
The matrices, $V_{0}$, $W_{0}$, $\sigma$, are unitary matrices and satisfy the conditions (QW1)-(QW4) with 
$T=\alpha I$, where $T$ is defined in $\eqref{shift1}$.
\end{example}

\begin{example}\label{ex1}
{\rm 
Let us consider the quantum walks $V=U(V_{0})$, $W=U(W_{0})$ defined by the formula $\eqref{qwalk}$ with the unitary matrix $A$ replaced by $V_{0}$, $W_{0}$ 
given in Example $\ref{ex0}$, respectively. 
Then, the unitary operators, $V$, $W$, $\sigma$, on $\ell^{2}(\mb{Z}) \otimes \mb{C}^{2}$ also satisfy the conditions (QW1)-(QW4) with $T=\alpha \tau$, 
where $\tau$ is, as in Section \ref{INTRO}, the shift operator. } 
\end{example}
Next, we impose the following condition. 


\begin{itemize}
\item[(QWR)] There exists a unit vector $e \in \pi_{+}\hcal_{0}$ such that $\ispa{V^{x}e,e}=0$ for any positive integer $x$. 
\end{itemize}


Note that the conditions (QW1)-(QW4) have the meaning for unitary elements, $V$, $W$, $\sigma$, 
in a $C^{*}$-algebra, say $\acal_{0}$. In this context, the condition (QWR) is a condition on a representation on $\hcal_{0}$ 
of the $C^{*}$-subalgebra in $\acal_{0}$ generated by $V$, $W$, $\sigma$. 

\begin{example}\label{ex2}
{\rm The unitary operator $V$ on $\ell^{2}(\mb{Z}) \otimes \mb{C}^{2}$ given in Example $\ref{ex1}$ satisfies the condition (QWR) with the unit vector $e=\delta_{0} \otimes {\bf e}_{1}$. 
However, the unitary matrix $V_{0}$ on $\mb{C}^{2}$ given in Example $\ref{ex0}$ does not satisfy the condition (QWR).}
\end{example}

In what follows, we suppose that the operator $V$ satisfies the condition (QWR) with a unit vector $e \in \pi_{+}\hcal_{0}$. 

\begin{lem}\label{ons1}
The vectors $T^{x}e$ $(x \in \mb{Z})$ form an orthonormal system in $\hcal_{0}$. 
\end{lem}

\begin{proof}
It is enough to show that $\ispa{T^{x}e,e}=0$ for any non-zero integer $x$. 
Since $V$ and $\sigma$ are commutative, we have 
\begin{equation}\label{TandV}
Te=Ve \in \pi_{+}\hcal_{0},\quad T^{-1}e=T^{*}e=V^{*}e \in \pi_{+}\hcal_{0}. 
\end{equation}
Since $T$ and $V$ are commutative, we see $T^{x}e=V^{x}e$ for any integer $x$, which combined with the assumption (QWR) shows the assertion. 
\end{proof}
We set 
\[
e_{1}^{x}=T^{x}e,\quad e_{2}^{x}=T^{x}\ep e \quad (x \in \mb{Z}),\quad 
e_{1}=e_{1}^{0},\quad e_{2}=e_{2}^{0}.  
\]
Then, by Lemma $\ref{auxL1}$ and the fact that $T$ preserves the eigenspaces of $\sigma$, we have $e_{1}^{x} \in \pi_{+}\hcal_{0}$, $e_{2}^{x} \in \pi_{-}\hcal_{0}$, 
and the vectors $e_{i}^{x}$ $(i=1,2,\ \,x \in \mb{Z})$ form an orthonormal system in $\hcal_{0}$. 
Now, let $\ell$ be the closure of the subspace spanned by $\{e_{1}^{x}\,;\,x \in \mb{Z}\}$. 
Then, the closure of the subspace spanned by $\{e_{2}^{x}\,;\,x \in \mb{Z}\}$ is $\ep \ell$. 
Since $\ell \subset \pi_{+}\hcal_{0}$, $\ep \ell \subset \pi_{-}\hcal_{0}$, these are orthogonal to each other. We define the subspace 
\[
\hcal:=\ell \oplus \ep \ell
\]
in $\hcal_{0}$, in which the set $\{e_{i}^{x}\,;\,i=1,2,\ x \in \mb{Z}\}$ becomes a complete orthonormal system. 
\begin{lem}\label{ons2}
The operators, $V$, $W$, $\sigma$, and their inverses preserve the subspace $\hcal$ so that these define unitary operators on $\hcal$. 
\end{lem}
Lemma $\ref{ons2}$ follows from $\eqref{TandV}$ and Lemma $\ref{auxL1}$. Indeed, we have the following. 
\begin{equation}\label{act1}
Ve_{1}^{x}=e_{1}^{x+1},\quad Ve_{2}^{x}=e_{2}^{x-1},\quad We_{1}^{x}=e_{2}^{x+1},\quad We_{2}^{x}=-e_{1}^{x-1}. 
\end{equation}
Our main interest is in the operator, $U$, defined by 
\begin{equation}\label{opU1}
U=sV+tW,\quad s,t \in \mb{R},\ \ s^{2}+t^{2}=1. 
\end{equation}
It is obvious that $U$ is a unitary operator on $\hcal_{0}$ and on $\hcal$. 
Since the operators of the form $\pm V$, $\pm W$, $\sigma$ also satisfy the conditions (QW1)-(QW4), (QWR), 
we may assume that $0 \leq s,t \leq 1$. 
Let $\psi =\,^{t}(\psi_{1},\psi_{2})$ be a unit vector in $\mb{C}^{2}$, 
and define $\Psi=\psi_{1}e_{1}+\psi_{2}e_{2}$, which is a unit vector in $\hcal$. 
For any integer $n$, we define a function $q_{n}(\psi;x)$ in $x \in \mb{Z}$ by 
\begin{equation}\label{df1}
q_{n}(\psi;x):=|\ispa{U^{n}\Psi,e_{1}^{x}}|^{2}+|\ispa{U^{n}\Psi,e_{2}^{x}}|^{2}. 
\end{equation}
Since $U$ is a unitary operator on $\hcal$, we have $\sum_{x \in \mb{Z}}q_{n}(\psi;x)=1$ and hence it defines a probability distribution on $\mb{Z}$. 

\begin{example}\label{ex3}
{\rm 
Let $V$, $W$, $\sigma$ be the unitary operators on $\ell^{2}(\mb{Z}) \otimes \mb{C}$ defined in Example $\ref{ex1}$. 
Let $e=\delta_{0} \otimes {\bf e}_{1}$, which satisfies the condition (QWR) as in Example $\ref{ex2}$. 
In this example, the unitary operator $U$ defined in $\eqref{opU1}$ has the form
\[
U=
\begin{pmatrix}
a & 0 \\
-\ol{b} & 0 
\end{pmatrix}\tau+
\begin{pmatrix}
0 & b \\
0 & \ol{a}
\end{pmatrix}\tau^{-1},\quad 
a=s\alpha,\ b=t\beta, 
\]
which coincides with the quantum walk $U(A)$ defined in $\eqref{qwalk}$ with the 
unitary matrix $A$ in $\eqref{Amat}$. 
Since $VW=
\begin{pmatrix}
0 & \alpha\beta \\
-\ol{\alpha}\ol{\beta} & 0
\end{pmatrix}$, we see $e_{2}=VW(\delta_{0} \otimes {\bf e}_{1})=-\ol{\alpha}\ol{\beta}(\delta_{0} \otimes {\bf e}_{2})$. 
Since $T=\alpha \tau$ as in Example $\ref{ex2}$, the vectors, $e_{i}^{x}$, are given by 
\[
e_{1}^{x}=\alpha^{x}(\delta_{x} \otimes {\bf e}_{1}),\quad 
e_{2}^{x}=-\alpha^{x-1}\ol{\beta}(\delta_{x} \otimes {\bf e}_{2}) \quad (x \in \mb{Z}). 
\]
Thus, we have $\hcal_{0}=\hcal=\ell^{2}(\mb{Z}) \otimes \mb{C}^{2}$. 
As above, let $\psi=\,^{t}(\psi_{1},\psi_{2})$ be a unit vector in $\mb{C}^{2}$. Then the corresponding vector $\Psi$ in $\ell^{2}(\mb{Z}) \otimes \mb{C}^{2}$ is 
\[
\Psi=\psi_{1}e_{1}+\psi_{2}e_{2}=\delta_{0} \otimes \varphi,\quad 
\varphi=\,^{t}(\psi_{1},-\ol{\alpha}\ol{\beta}\psi_{2}).  
\]
From this we have 
\[
\begin{gathered}
\ispa{U^{n}\Psi,e_{1}^{x}}=\alpha^{x}\ispa{U^{n}(\delta_{0} \otimes \varphi),\delta_{x} \otimes {\bf e}_{1}},\\
\ispa{U^{n}\Psi,e_{2}^{x}}=-\alpha^{-x+1}\beta\ispa{U^{n}(\delta_{0} \otimes \varphi), \delta_{x} \otimes {\bf e}_{2}},  
\end{gathered}
\]
which implies $q_{n}(\psi;x)=p_{n}^{A}(\varphi;x)$ with $(\psi_{1},\psi_{2})=(\varphi_{1},-\alpha \beta \varphi_{2})$, where $p_{n}^{A}(\varphi;x)$ is defined in $\eqref{df0}$. 
}
\end{example}

As in Example $\ref{ex3}$, it is enough to prove the following theorem to show Theorem $\ref{KonnoWL}$.

\begin{thm}\label{WLT}
Let $\psi=\,^{t}(\psi_{1},\psi_{2})$ be a unit vector in $\mb{C}^{2}$. Let $s,t$ be real numbers satisfying $0 < s,t < 1$, $s^{2}+t^{2}=1$.  Then we have 
\[
\wlim_{n \to \infty}\sum_{x \in \mb{Z}}q_{n}(\psi;x)\delta_{x/n}=\chi_{(-s,s)}(y)
\frac{t(1+\lambda(\psi;s,t)y)}{\pi (1-y^{2})\sqrt{s^{2}-y^{2}}}\,dy, 
\]
where the function $\lambda(\psi;s,t)$ is given by 
\[
\lambda(\psi;s,t)=|\psi_{1}|^{2}-|\psi_{2}|^{2}+2
\re(\psi_{1}\ol{\psi}_{2})\frac{t}{s}.
\]
\end{thm}
We give a simple proof of Theorem $\ref{WLT}$ in the following two sections.

\section{Computation of the distribution $q_{n}(\psi;x)$}\label{COMP}

The unitary operator, $U$, defined in $\eqref{opU1}$ is written as $U=x+iy+w$ with $x=sX$, $y=sY$, $w=tW$, where $X$, $Y$ are 
self-adjoint operators defined in $\eqref{shift1}$. By  $\eqref{rel1}$, we see that the operators $x$ and $iy+w$ are commutative and  
\begin{equation}\label{rel2}
(iy+w)^{2}=-(y^{2}+t^{2}),\quad x^{2}+y^{2}+t^{2}=I. 
\end{equation}
Therefore, we get the following. 
\[
\begin{split}
U^{n} & = (x+(iy+w))^{n}=\sum_{k=0}^{n}\binom{n}{k}x^{n-k}(iy+w)^{k} \\
& = \sum_{l=0}^{[n/2]}\binom{n}{2l}x^{n-2l}(x^{2}-I)^{l}+\sum_{l=0}^{[(n-1)/2]}\binom{n}{2l+1}x^{n-2l-1}(x^{2}-I)^{l}(iy+w). 
\end{split}
\]
Next, we recall a definition of the Chebyshev polynomials (see \cite{R} for details). 
The Chebyshev polynomials of the first kind, $T_{n}(x)$ (of degree $n$), and the second kind, $U_{n-1}(x)$ (of degree $n-1$), are defined as 
\[
T_{n}(\cos \theta)=\cos n\theta,\quad U_{n-1}(\cos \theta)=\frac{\sin n \theta}{\sin \theta}. 
\]
These polynomials can be written in the following form. 
\[
T_{n}(x)=\sum_{l=0}^{[n/2]}\binom{n}{2l}x^{n-2l}(x^{2}-1)^{l},\quad 
U_{n-1}(x)=\sum_{l=0}^{[(n-1)/2]} \binom{n}{2l+1}x^{n-2l-1}(x^{2}-1)^{l}. 
\]
From this we can write $U^{n}$ as 
\begin{equation}\label{UN}
U^{n}=T_{n}(x)+U_{n-1}(x)(iy+w)=T_{n}(x)+(iy+w)U_{n-1}(x),  
\end{equation}
where $x$, $y$ and $w$ are operators defined at the beginning of this section. 
A direct computation shows that 
\[
\begin{gathered}
xe_{1}=\frac{s}{2}(T+T^{-1})e_{1},\quad xe_{2}=\frac{s}{2}(T+T^{-1})e_{2},\\
iye_{1}=\frac{s}{2}(T-T^{-1})e_{1},\quad iye_{2}=-\frac{s}{2}(T-T^{-1})e_{2}, \quad 
we_{1}=tTe_{2},\quad we_{2}=-tT^{-1}e_{1}. 
\end{gathered}
\]
Let $p_{n}^{i}(z)$, $q_{n}^{i}(z)$ $(i=1,2)$ be Laurent polynomials (with real coefficients) defined by 
\begin{equation}\label{cefP}
\begin{split}
p_{n}^{1}(z) & = T_{n}(s(z+z^{-1})/2)+\frac{s}{2}(z-z^{-1})U_{n-1}(s(z+z^{-1})/2),\\
p_{n}^{2}(z) & = tzU_{n-1}(s(z+z^{-1})/2),\\
q_{n}^{1}(z) & = -tz^{-1}U_{n-1}(s(z+z^{-1})/2),\\
q_{n}^{2}(z) & = T_{n}(s(z+z^{-1})/2)-\frac{s}{2}(z-z^{-1})U_{n-1}(s(z+z^{-1})/2). 
\end{split}
\end{equation}
Since $T$ commutes with all the operators, we have 
\begin{equation}\label{powerN}
\begin{split}
U^{n}e_{1} & = p_{n}^{1}(T)e_{1} +p_{n}^{2}(T)e_{2}, \\
U^{n}e_{2} & = q_{n}^{1}(T)e_{1} +q_{n}^{2}(T)e_{2}.
\end{split}
\end{equation}
For any Laurent polynomial $p(z)$ in $z \in \mb{C}$, let us denote the coefficient of $z^{x}$ ($x \in \mb{Z}$) in $p(z)$ by $\ct_{x}(p)$. Then, 
for any $x \in \mb{Z}$, $\eqref{powerN}$ shows the following. 
\begin{equation}\label{powerC}
\begin{split}
\ispa{U^{n}\Psi,e_{1}^{x}} & = \psi_{1} \ct_{x}(p_{n}^{1})+\psi_{2}\ct_{x}(q_{n}^{1}),\\
\ispa{U^{n}\Psi,e_{2}^{x}} & = \psi_{1} \ct_{x}(p_{n}^{2})+\psi_{2}\ct_{x}(q_{n}^{2}). 
\end{split}
\end{equation}
Since $p_{n}^{i}(z)$, $q_{n}^{i}(z)$ have real coefficients, we have obtained the following lemma. 
\begin{lem}\label{EXP1}
The probability distribution $q_{n}(\psi;x)$ on $\mb{Z}$ defined in $\eqref{df1}$ is written as 
\begin{equation}\label{df2}
\begin{split}
q_{n}(\psi;x) & = |\psi_{1}|^{2}\left[
\ct_{x}(p_{n}^{1})^{2}+\ct_{x}(p_{n}^{2})^{2}
\right] \\
& \hspace{40pt} +|\psi_{2}|^{2}\left[
\ct_{x}(q_{n}^{1})^{2}+\ct_{x}(q_{n}^{2})^{2}
\right] \\
& \hspace{20pt} +2\re(\psi_{1}\ol{\psi}_{2})\left[
\ct_{x}(p_{n}^{1})\ct_{x}(q_{n}^{1})+\ct_{x}(p_{n}^{2})\ct_{x}(q_{n}^{2})
\right]. 
\end{split}
\end{equation}
In particular, the characteristic function
\[
E_{n}(\xi)=\sum_{x \in \mb{Z}}q_{n}(\psi;x)e^{i\xi x},\quad \xi \in \mb{R}, 
\]
is given by 
\begin{equation}\label{cha1}
\begin{split}
E_{n}(\xi)  = |\psi_{1}|^{2} & \left[
P_{n}^{1}(\xi)+P_{n}^{2}(\xi)\right] +|\psi_{2}|^{2}
\left[
Q_{n}^{1}(\xi)+Q_{n}^{2}(\xi)\right] \\
& +2\re(\psi_{1}\ol{\psi}_{2})\left[
R_{n}^{1}(\xi)+R_{n}^{2}(\xi)
\right], 
\end{split}
\end{equation}
where the functions $P_{n}^{i}(\xi)$, $Q_{n}^{i}(\xi)$ and $R_{n}^{i}(\xi)$ $(i=1,2)$ are given by 
\begin{equation}
\begin{gathered}
P_{n}^{i}(\xi) = \sum_{x \in \mb{Z}}\ct_{x}(p_{n}^{i})^{2}e^{i\xi x},\quad 
Q_{n}^{i}(\xi) = \sum_{x \in \mb{Z}}\ct_{x}(q_{n}^{i})^{2}e^{i\xi x},\\
R_{n}^{i}(\xi) = \sum_{x \in \mb{Z}}\ct_{x}(p_{n}^{i})\ct_{x}(q_{n}^{i})e^{i\xi x}. 
\end{gathered}
\end{equation}
\end{lem}

\section{Proof of the weak limit theorem}\label{PROOF}

Let $p(z)$, $q(z)$ be two Laurent polynomials, which are, in our notation, written as 
\[
p(z)=\sum_{x \in \mb{Z}}\ct_{x}(p)z^{x},\quad q(z)=\sum_{x \in \mb{Z}}\ct_{x}(q)z^{x}. 
\]
Then, we have the following convolution identity. 
\begin{equation}\label{cross1}
\sum_{x \in \mb{Z}}\ct_{x}(p)\ct_{x}(q)w^{x}=\int_{|z|=1} p(wz)q(z^{-1})\,\frac{dz}{2\pi i z}. 
\end{equation}
We use the formula $\eqref{cross1}$ to deduce the asymptotic properties of $P_{n}^{i}(\xi/n)$, $Q_{n}^{i}(\xi/n)$ and $R_{n}^{i}(\xi/n)$ as $n$ tends to infinity. 
To compute it, we need the following lemmas. 
\begin{lem}\label{taylor1}
For $\theta \in \mb{R}$, we set $f(\theta)={\rm Cos}^{-1}(s\cos \theta)$, where $0<s<1$. We fix $\xi \in \mb{R}$. 
Then, as $n \to \infty$, we have the following. 
\[
\begin{split}
T_{n}(s \cos(\theta +\xi/n)) & =T_{n}(s \cos \theta)  \cos (\xi f'(\theta)) \\
&\hspace{30pt} -\sin (n f(\theta)) \sin (\xi f'(\theta)) +O(1/n),\\
U_{n-1}(s \cos (\theta +\xi/n)) & =U_{n-1}(s \cos \theta) \cos (\xi f'(\theta)) \\
& \hspace{30pt} +T_{n}(s \cos \theta) \frac{\sin (\xi f'(\theta))}{\sin f(\theta)} +O(1/n),  
\end{split}
\]
where $O(1/n)$ is uniform in $\theta$. 
\end{lem}

\begin{proof}
Noting that $T_{n}(s \cos (\theta +\xi/n))=\cos (n f(\theta +\xi/n))$ and $0<{\rm Cos}^{-1}(s) \leq f(\theta) \leq {\rm Cos}^{-1}(-s)<\pi$, 
a simple application of the Taylor expansion shows the lemma. 
\end{proof}

\begin{lem}\label{asym1}
Let $k \in \mb{Z}$ and $\xi \in \mb{R}$. We set
\[
\begin{split}
A_{n,k}(\xi) & = \int_{|z|=1}z^{k}T_{n}(s(e^{i\xi/n}z+e^{-i\xi/n}z^{-1})/2)T_{n}(s(z+z^{-1})/2)\,\frac{dz}{2\pi i z},\\
B_{n,k}(\xi) & = \int_{|z|=1}z^{k}T_{n}(s(e^{i\xi/n}z+e^{-i\xi/n}z^{-1})/2)U_{n-1}(s(z+z^{-1})/2)\,\frac{dz}{2\pi i z}, \\
C_{n,k}(\xi) & = \int_{|z|=1}z^{k}U_{n-1}(s(e^{i\xi/n}z+e^{-i\xi/n}z^{-1})/2) T_{n}(s(z+z^{-1})/2)\,\frac{dz}{2\pi i z}, \\
D_{n,k}(\xi) & = \int_{|z|=1}z^{k}U_{n-1}(s(e^{i\xi/n}z+e^{-i\xi/n}z^{-1})/2) U_{n-1}(s(z+z^{-1})/2)\,\frac{dz}{2\pi i z}. 
\end{split}
\]
Then, we have the following. 
\[
\begin{split}
\lim_{n \to \infty}A_{n,k}(\xi) & = \frac{1+(-1)^{k}}{4\pi} 
\int_{-s}^{s}e_{k}(x) \cos
\left(
\frac{\xi\sqrt{s^{2}-x^{2}}}{\sqrt{1-x^{2}}}
\right)
\frac{dx}{\sqrt{s^{2}-x^{2}}}, \\
\lim_{n \to \infty}B_{n,k}(\xi) & = -\lim_{n \to \infty}C_{n,k}(\xi) \\
& = -\frac{1-(-1)^{k}}{4\pi}\int_{-s}^{s}
e_{k}(x) \sin 
\left(
\frac{\xi \sqrt{s^{2}-x^{2}}}{\sqrt{1-x^{2}}}
\right)\frac{dx}{\sqrt{(1-x^{2})(s^{2}-x^{2})}},\\
\lim_{n \to \infty}D_{n,k}(\xi) & = 
\frac{1+(-1)^{k}}{4\pi}\int_{-s}^{s}
e_{k}(x) \cos
\left(
\frac{\xi \sqrt{s^{2}-x^{2}}}{\sqrt{1-x^{2}}} 
\right)
\frac{dx}{(1-x^{2})\sqrt{s^{2}-x^{2}}}.
\end{split}
\]
where we set $e_{k}(x)=\exp[i k {\rm Cos}^{-1}(x/s)]$. 
\end{lem}

\begin{proof}
Since $T_{n}(-x)=(-1)^{n}T_{n}(x)$, $U_{n-1}(-x)=(-1)^{n-1}U_{n-1}(x)$, we see 
\[
\begin{split}
A_{n,k}(\xi) & = \frac{1+(-1)^{k}}{2\pi} \int_{0}^{\pi}e^{ik \theta}T_{n}(s \cos (\theta +\xi/n))T_{n}(s \cos \theta)\,d\theta, \\
B_{n,k}(\xi) & = \frac{1-(-1)^{k}}{2\pi} \int_{0}^{\pi}e^{ik\theta} T_{n}(s \cos (\theta +\xi/n)) U_{n-1}(s \cos \theta)\,d\theta,\\
C_{n,k}(\xi) & = \frac{1-(-1)^{k}}{2\pi} \int_{0}^{\pi}e^{ik\theta} U_{n-1}(s \cos (\theta +\xi/n)) T_{n}(s \cos \theta)\,d\theta,\\
D_{n,k}(\xi) & = \frac{1+(-1)^{k}}{2\pi} \int_{0}^{\pi}e^{ik\theta} U_{n-1}(s \cos (\theta +\xi/n)) U_{n-1}(s \cos \theta)\,d\theta. 
\end{split}
\]
Note that we have 
\[
\int_{0}^{\pi}|T_{n}(s \cos \theta)| \,d\theta \leq \pi,\quad \int_{0}^{\pi}|U_{n-1}(s \cos \theta)|\,d\theta \leq \frac{\pi}{\sqrt{1-s^{2}}}. 
\]
Thus, by Lemma $\ref{taylor1}$ and the change of the variable, $\varphi=f(\theta)$, where $f(\theta)$ is the function introduced in Lemma $\ref{taylor1}$, we have 
\[
\begin{split}
A_{n,k}(\xi) 
& = \frac{1+(-1)^{k}}{2\pi}
\left[
\int_{{\rm Cos}^{-1}(s)}^{{\rm Cos}^{-1}(-s)}
e_{k}(\cos \varphi)\cos^{2}(n\varphi) \right. \\
& \hspace{35mm} \left. \times \cos
\left(
\frac{\xi\sqrt{s^{2}-\cos^{2}\varphi}}{\sin \varphi}
\right)\frac{\sin \varphi}{\sqrt{s^{2}-\cos^{2}\varphi}}\,d\varphi
\right. \\
&\left. - 
\int_{{\rm Cos}^{-1}(s)}^{{\rm Cos}^{-1}(-s)}
e_{k}(\cos \varphi)\sin (n\varphi)\cos (n\varphi) \right. \\
& \hspace{35mm} \left. \times \sin 
\left(
\frac{\xi\sqrt{s^{2}-\cos^{2}\varphi}}{\sin \varphi}
\right)\frac{\sin \varphi}{\sqrt{s^{2}-\cos^{2}\varphi}}\,d\varphi+O(1/n)
\right].  
\end{split}
\]
Since $\frac{1}{\sqrt{s^{2}-\cos^{2}\varphi}}$ is an $L^{1}$-function on the integration interval, the Riemann-Lebesgue lemma gives 
\[
\begin{split}
\lim_{n \to \infty}A_{n,k}(\xi) & = \frac{1+(-1)^{k}}{4\pi}
\int_{{\rm Cos}^{-1}(s)}^{{\rm Cos}^{-1}(-s)}
e_{k}(\cos \varphi) \\
& \hspace{35mm} \times \cos
\left(
\frac{\xi\sqrt{s^{2}-\cos^{2}\varphi}}{\sin \varphi}
\right)\frac{\sin \varphi}{\sqrt{s^{2}-\cos^{2}\varphi}}\,d\varphi \\
& = \frac{1+(-1)^{k}}{4\pi}
\int_{-s}^{s}
e_{k}(x)
\cos
\left(
\frac{\xi\sqrt{s^{2}-x^{2}}}{\sqrt{1-x^{2}}}
\right)\frac{dx}{\sqrt{s^{2}-x^{2}}}.
\end{split}
\]
Formulas for $B_{n,k}(\xi)$, $C_{n,k}(\xi)$ and $D_{n,k}(\xi)$ can be obtained in a similar way. 
\end{proof}

Now, it is rather easy to prove Theorem $\ref{WLT}$. 

\vspace{10pt}

\noindent{\bf Proof of Theorem $\ref{WLT}$.} \hspace{3pt} A direct computation using $\eqref{cross1}$ and Lemma $\ref{asym1}$ leads us to the following. 
\[
\begin{split}
P_{n}^{1}(\xi/n) & = A_{n,0}(\xi)-\frac{s}{2}\left(
B_{n,1}(\xi) -B_{n,-1}(\xi)\right)+\frac{s}{2}\left(
e^{i\xi/n}C_{n,1}(\xi) -e^{-i\xi/n}C_{n,-1}(\xi)
\right) \\
& \hspace{20pt} -\frac{s^{2}}{4}
\left(
e^{i\xi/n}D_{n,2}(\xi) -2\cos(\xi/n)D_{n,0}(\xi)+e^{-i\xi/n}D_{n,-2}(\xi)
\right) \\
& \to \frac{1}{2\pi} \int_{-s}^{s} 
\cos \left(
\frac{\xi \sqrt{s^{2}-x^{2}}}{\sqrt{1-x^{2}}}
\right) \frac{dx}{\sqrt{s^{2}-x^{2}}} \\
& \hspace{20pt}+\frac{i}{\pi} \int_{-s}^{s}
\sin \left(
\frac{\xi \sqrt{s^{2}-x^{2}}}{\sqrt{1-x^{2}}}
\right) \frac{dx}{\sqrt{1-x^{2}}} \\
& \hspace{20pt} +\frac{1}{2\pi} \int_{-s}^{s}
\cos\left(
\frac{\xi \sqrt{s^{2}-x^{2}}}{\sqrt{1-x^{2}}}
\right)\frac{\sqrt{s^{2}-x^{2}}}{1-x^{2}}\,dx \qquad (n \to \infty). 
\end{split}
\]
Note that the integrand are all even function. Changing the variable $y=\sqrt{\frac{s^{2}-x^{2}}{1-x^{2}}}$, we have 
\[
\lim_{n \to \infty}P_{n}^{1}(\xi/n)=
\frac{t}{\pi}
\int_{0}^{s}\{(1+y^{2})\cos (\xi y)+2i y \sin (\xi y)\} \frac{dy}{(1-y^{2})\sqrt{s^{2}-y^{2}}}.
\]
Similar computation shows that
\[
\begin{split}
\lim_{n \to \infty}Q_{n}^{1}(\xi/n) & = \lim_{n \to \infty}P_{n}^{2}(\xi/n) = \frac{t}{\pi}\int_{0}^{s}\cos (\xi y) \frac{dy}{\sqrt{s^{2}-y^{2}}}, \\
\lim_{n \to \infty}Q_{n}^{2}(\xi/n) & = \frac{t}{\pi} \int_{0}^{s} \{(1+y^{2})\cos (\xi y) -2iy \sin (\xi y)\} \\
& \hspace{35mm} \times \frac{dy}{(1-y^{2})\sqrt{s^{2}-y^{2}}}, \\
\lim_{n \to \infty}(R_{n}^{1}(\xi/n)+R_{n}^{2}(\xi/n)) & = \frac{2t^{2} i}{\pi s} \int_{0}^{s} y \sin (\xi y) \frac{dy}{(1-y^{2})\sqrt{s^{2}-y^{2}}}. 
\end{split}
\]
Noting that integrals of odd functions on the interval $(-s,s)$ vanish, 
we have, by Lemma $\ref{EXP1}$, 
\[
\lim_{n \to \infty}E_{n}(\xi/n)=\frac{t}{\pi} \int_{-s}^{s}e^{i\xi y}
\frac{1+\lambda(\psi;s,t)y}{(1-y^{2})\sqrt{s^{2}-y^{2}}}\,dy, 
\]
which completes the proof of Theorem $\ref{WLT}$. \hfill$\square$

\section*{Acknowledgment} 
The author would like to thank the referee for invaluable comments.

\end{document}